\documentclass[12pt]{amsart}
\pdfoutput=1
\usepackage[margin=1in]{geometry}
\usepackage{amsmath,amsthm,amssymb,amsrefs,bbm,color,enumitem,esint,esvect,float,graphicx,mathrsfs}
\usepackage{euscript,latexsym}
\usepackage{accents, xcolor}

\usepackage{epstopdf}
\AppendGraphicsExtensions{.tif}

\usepackage{bbm}

\usepackage{hyperref}
\hypersetup{colorlinks}

\newcommand{\R}{\mathbb{R}}

\newcommand{\N}{\mathbb{N}}
\newcommand{\Z}{\mathbb{Z}}

\newcommand{\D}{\mathcal{D}}

\newcommand{\BMO}{\mathrm{BMO}}
\newcommand{\Hone}{\mathrm{H}^1}
\newcommand{\Lip}{\Lambda_q(\alpha)}
\newcommand{\cexp}{\mathsf{E}}
\newcommand{\diff}{\mathsf{D}}

\newtheorem{thm}{Theorem}[section]
\newtheorem{lemma}[thm]{Lemma}

\newtheorem{prop}[thm]{Proposition}
\newtheorem{rem}[thm]{Remark}

\newtheorem{open}[thm]{Open Question}

\newtheorem{ltheorem}{Theorem}

\numberwithin{equation}{section}

\def\one{\mbox{1\hspace{-4.25pt}\fontsize{12}{14.4}\selectfont\textrm{1}}}

\theoremstyle{definition}

\begin{document}

\title[Endpoint Estimates]{Endpoint Estimates for Haar Shift Operators with Balanced Measures}
\author{Jos\'{e} M. Conde Alonso}
\address{Jos\'{e} M. Conde Alonso \hfill\break\indent 
 Departamento de Matem\'aticas \hfill\break\indent 
 Universidad Aut\'onoma de Madrid \hfill\break\indent 
 C/ Francisco Tom\'as y Valiente sn\hfill\break\indent 
 28049 Madrid, Spain}
\email{jose.conde@uam.es}
\author[N.A. Wagner]{Nathan A. Wagner}
\thanks{J. M. Conde-Alonso was supported by grant \texttt{CNS2022-135431} (Agencia Estatal de Investigaci\'on, Spain). N. A. Wagner was supported in part by National Science Foundation grant DMS \texttt{220327}2.}

\address{Nathan A. Wagner \hfill\break\indent 
 Department of Mathematics \hfill\break\indent 
Brown University \hfill\break\indent 151 Thayer St \hfill\break\indent 
Providence, RI, 02912 USA}
\email{nathan\_wagner@brown.edu}


\subjclass[2010]{}%

\keywords{}%

\begin{abstract} 
We prove $\Hone$ and $\BMO$ endpoint inequalities for generic cancellative Haar shifts defined with respect to a possibly non-homogeneous Borel measure $\mu$ satisfying a weak regularity condition. This immediately yields a new, highly streamlined proof of the $L^p$-results for the same operators due to L\'opez-Sanchez, Martell, and Parcet \cite{LSMP}. We also prove regularity properties for the Haar shift operators on the natural martingale Lipschitz spaces defined with respect to the underlying dyadic system, and show that the class of measures that we consider is sharp.
\end{abstract}

\maketitle
\section*{Introduction}

Let $\D$ be a dyadic system on $\R^n$. By that we mean a family
$$
\D = \bigcup_{k\in\Z} \D_k = \bigcup_{k\in\Z, j \in \Z^n} \{Q_{jk}\},
$$
where each $Q_{jk}$ is a half open cube with sides parallel to the axes and satisfying the following properties:
\begin{itemize}
    \item The side length $\ell(Q_{jk})$ of $Q_{jk}$ is $2^{-k}$.
    \item The $k$-th generation $\D_k = \cup_{j \in \Z^n} \{Q_{jk}\}$ is a partition of $\R^n$.
    \item Any two cubes in $\D$ are either disjoint or nested.
\end{itemize}
Given $r>0$ and $Q\in\D$, we denote by $\D_r(Q)$ the family of cubes $R\in\D$ such that $R\subset Q$ and $\ell(Q) = 2^r \ell(R)$. To each $Q\in \D$, one can associate up to $2^n-1$ Haar functions $\{h_Q^j\}_{1\leq j < 2^{n}}$ supported on $Q$, which satisfy the following: they have integral $0$, $h_Q^j$ is constant over each $R\in \D_1(Q)$ and the family $\{h_Q^j\}_{Q\in\D,\,j}$ is an orthonormal basis of $L^2(\R^n)$. A (cancellative) Haar shift operator of complexity $(r,s) \in \N\times\N$ is an operator of the form
$$
T^{r,s} f(x)= \sum_{Q \in \D}\sum_{\substack{R\in\D_r(Q) \\ S\in \D_s(Q)}} \sum_{j,j'}\alpha_{R,S}^Q \langle f, h_{R}^j \rangle h_{S}^{j'}(x). 
$$
Haar shifts are some of the most succesful discrete models of Calder\'on-Zygmund operators. Indeed, the groundbreaking result of Petermichl \cite{Petermichl2000} allows one to write the Hilbert transform
$$
\mathcal{H}f(x) := \mathrm{p.v.} \int_\R \frac{f(y)}{x-y} dy
$$
as an average of the same dyadic shift defined over dilations and translations of the usual dyadic grid. Said shift is defined by
$$
\mathcal{H}_\D f(x) := \sum_{I \in \D} \langle f, h_{I} \rangle \left(h_{I_-}(x)-h_{I_+}(x)\right),
$$
with respect to each dyadic system $\D$. Above $I_-$ and $I_+$ are the two elements of $\D_1(I)$. A higher dimensional, remarkable generalization is available: any Calder\'on-Zygmund operator on $\R^n$ can be represented as an appropriate average of Haar shift operators, possibly more complicated than $\mathcal{H}_\D$, see \cite{Hyt2012}. These representations motivate the detailed study of the simpler models $T^{r,s}$, especially in what pertains their precise quantitative behavior (for example, in terms of weights).

The classical Calder\'on-Zygmund techniques can be applied to study Haar shift operators. The $L^2$-theory is significantly simpler than that of singular integral operators: indeed, the orthogonality of the Haar system implies
$$
\left\|T^{r,s}\right\|_{L^2(\R^n) \to L^2(\R^n)} \lesssim_{r,s,n} \left\|\left\{ \alpha_{R,S}^Q\right\}\right\|_{\ell^\infty},
$$
and the boundedness of the coefficients $\alpha_{R,S}^Q$ is necessary in general. $L^2$-bounded Haar shifts can then be shown to map $L^p(\R^n)$ onto itself and to be of weak-type $(1,1)$. 

There are numerous occasions when one may need to consider a measure on $\R^n$ different from the Lebesgue measure. The very definition of Haar shifts is affected by any change on $\mu$, because the orthogonality of the elements of the Haar basis must account for it. In the one-dimensional case, where up to sign there is only one possible choice of Haar basis element per dyadic interval, for each $I\in \D$ we have now
$$
h_I(x) = \sqrt{\frac{\mu(I_-)\mu(I_+)}{\mu(I)}}\left(\frac{\one_{I_-}(x)}{\mu(I_-)}-\frac{\one_{I_+}}{\mu(I_+)}\right).
$$
Using the same techniques as in the Lebesgue measure case, one can show that $T^{r,s}$ are $L^p$-bounded and of weak type $(1,1)$ if $\mu$ is dyadically doubling, that is, if $\mu(R)\sim \mu(Q)$ for all $Q\in\D$ and all $R\in\D_1(Q)$. But this is not the only situation: in \cite{LSMP}, it is shown that the condition
\begin{equation}\label{eq:balancedDimN}\tag{Bal}
    \sup_{Q\in \D}\sup_{R \in \D_1(Q)} \left[\|h_Q\|_{L^1(\mu)}\|h_R\|_{L^\infty(\mu)}+\|h_Q\|_{L^\infty(\mu)}\|h_R\|_{L^1(\mu)}\right] < \infty
\end{equation}
is the right one to ensure --meaning that it characterizes, up to non-degeneracy-- the $L^p$-boundedness and the weak-type $(1,1)$ bound of $T^{r,s}$. Measures satisfying \eqref{eq:balancedDimN} are called balanced. 

In this note, we address the problem of the boundedness of Haar shift operators defined with respect to general Borel-regular measures $\mu$ at two endpoints that had not been previously considered in the literature: the $L^\infty$ one and its predual. To that end, we consider the natural $\Hone$ and $\BMO$ spaces in the dyadic setting: if we set
$$
\cexp_k(f) = \sum_{Q \in \D_k} \langle f \rangle_Q \one_Q, \quad \diff_k(f)=\cexp_k(f)-\cexp_{k-1}(f),
$$
then we can define the martingale Hardy and $\BMO$ norms via
$$
\|f\|_{\Hone}=\left\|\mathcal{S}f\right\|_1 =\left\|\left(\sum_{k\in\Z} \left|\diff_k (f) \right|^2\right)^{\frac{1}{2}}\right\|_1 \quad \mbox{and} \quad \|f\|_{\BMO}=\sup_{k\in\Z} \left\|\cexp_k\left|f-\cexp_{k-1}(f)\right|\right\|_\infty,
$$
and we denote the corresponding Banach spaces by $\Hone(\mu)$ and $\BMO(\mu)$. Our first result is the following:

\begin{ltheorem}\label{th:theoremA}
Let $\mu$ be a balanced measure. If $T^{r,s}$ is a Haar shift operator of complexity $(r,s)$ which is bounded on $L^2(\mu)$, then:
\begin{itemize}
    \item[(i)] $T$ can be extended to a bounded operator on $\Hone(\mu)$.
    \item[(ii)] $T$ can be extended to a bounded operator on $\BMO(\mu)$.
\end{itemize}
\end{ltheorem}
Clearly, item (i) above implies (ii) because $[\Hone(\mu)]^*\simeq \BMO(\mu)$, but we have chosen to provide a direct argument for each statement. Moreover, we provide separate --and different-- proofs of the weaker inequalities $\Hone(\mu)\to L^1(\mu)$ and $L^\infty(\mu)\to \BMO(\mu)$. In harmonic analysis terms, the fact that we obtain the stronger endpoint results is natural: a Haar shift operator is a model of a Calder\'on-Zygmund operator $T$ that enjoys an extra cancellation, equivalent to $T1=0$ in $\BMO$. And that kind of Calder\'on-Zygmund operators do map $\BMO$ to $\BMO$. The $\Hone$-type inequality of Theorem \ref{th:theoremA}  uses the atomic block decomposition of $\Hone(\mu)$ from \cite{CAP}, and as far as we know the proof is new even in the Lebesgue measure case. We also remark that Theorem \ref{th:theoremA} provides a new proof of the $L^p$ bounds in \cite{LSMP} via interpolation which does not use the Calder\'{o}n-Zygmund decomposition.

Our second result concerns boundedness of Haar shift operators on martingale Lipschitz spaces. These are defined (see \cite{Weisz}*{page 7}) by the following semi-norm: given $1 \leq q< \infty$ and $\alpha>0$, we set
$$ 
\|f\|_{\Lip}:= \sup_{Q \in \D} \frac{1}{\mu(Q)^{1/q+\alpha}} \left( \int_Q |f- \langle f \rangle_{\widehat{Q}}|^q \, d\mu \right)^{1/q}<\infty.
$$
We define $\Lip$ as the family of functions functions $f \in L^q_{\text{loc}}(\mu)$ with finite $\|\cdot\|_{\Lip} $ norm. Making $\alpha=0$ allows one to recover $\BMO(\mu)$ regardless of the parameter $q$, by the John-Nirenberg inequality. On the other hand, $\Lip$ spaces are dual to $\mathrm{H}^p$ spaces for values $p=p(\alpha)<1$ \cite{Herz1973,Weisz} when $q=2$.

\begin{ltheorem}\label{th:theoremB}
Let $\mu$ be a balanced measure, and suppose $1 < q<\infty$ and $\alpha>0$. If $T^{r,s}$ is a Haar shift operator of complexity $(r,s)$ which is bounded on $L^2(\mu)$, then
 $$
 \|T f\|_{\Lip} \lesssim \|f\|_{\Lip},
 $$ 
 for all $f \in  \Lip$.
\end{ltheorem}

One can see Theorem \ref{th:theoremB} as a natural counterpart of the results by \cite{GG} in the continuous setting, where no conditions on the underlying measure are assumed. We shall show in contrast that our assumption is necessary in the dyadic setting. The rest of this note is organized as follows: Theorem \ref{th:theoremA} is proved in Section \ref{sec1}, while Section \ref{sec2} is devoted to Theorem \ref{th:theoremB}, and toward the end we collect several concluding remarks and open questions. 

\section{Endpoint results}\label{sec1}
To ease the exposition, we prove all results in dimension $n=1$. Also, for definiteness, we work with the standard dyadic lattice on $\R$
$$
\D:=\{2^{-k}[\ell,\ell+1): k, \ell \in \Z\}
$$ 
We always assume that $\mu$ is a balanced measure on $\R$ which satisfies $0<\mu(I)<\infty$ for all $I \in \D$, and also $\mu([0, \infty))= \mu(-\infty,0]=\infty$ as in \cite{LSMP}. This is not essential but it simplifies a bit our computations. The only adjustments needed in case one has quadrants of finite measure is to consider the corresponding indicator functions as elements of $\Hone(\mu)$ and as non-zero norm elements of $\BMO(\mu)$. For a detailed discussion of martingale norms over two-sided, non-homogeneous filtrations, one can see \cite{Treil13} (even though our situation is much less complicated than the general one studied there, and that is why we omit the details).

$\mathcal{D}_k$ denotes the collection of all dyadic intervals of length $2^{-k}$. If $I$ is a dyadic interval, we denote by $\widehat{I}$ its dyadic parent. In the one-dimensional situation, \eqref{eq:balancedDimN} reduces to 
\begin{equation}\label{eq:balancedDim1}
    m(I) := \min\left\{\mu(I_-),\mu(I_+)\right\} \sim \min\left\{\mu(\widehat{I}_-),\mu(\widehat{I}_+)\right\}.
\end{equation}
for all $I\in\D$. Given one such $I$, we enumerate the $2^m$ elements in $\D_m(I)$ by $I_m^s,  0 \leq s < 2^m.$ Without loss of generality (see \cite{CAPW}), we may suppose that a generic Haar shift $T$ takes the form
$$
Tf(x)= \sum_{I \in \D} \alpha_I \langle f, h_{I_{m}^s} \rangle h_{I_n^t}(x)  , \quad \sup_{I \in \D}|\alpha_I|\leq 1, 
$$
and below we always assume that our Haar shift operator is of that form.

\subsection{$\BMO$ estimates} We start proving estimates that hold for $f$ in a restricted class, and not the whole $\BMO(\mu)$. Below, we use the characterization of the norm given by
$$ 
\|f\|_{\BMO} \sim \sup_{I \in \D} \langle |f- \langle f \rangle_{I}| \rangle+ \sup_{I \in \D} |\langle f \rangle_{\widehat{I}}- \langle f \rangle_{I}|,
$$
which is straightforward to check.

\begin{thm}\label{LInfBMO}
If $f\in L^\infty(\mu)$ has compact support, then

\begin{equation*}
\|T f\|_{\BMO} \lesssim \|f\|_{L^\infty(\mathbb{R})}.  
\end{equation*}

\begin{proof} 
We have
$$
\|Tf\|_{\BMO} \sim \sup_{I \in \D} \langle |Tf- \langle Tf \rangle_{I}| \rangle_{I}+ \sup_{I \in \D} |\langle Tf \rangle_{\widehat{I}}- \langle Tf \rangle_{I}| =: \mathrm{A}+\mathrm{B},
$$
and we need to bound both of these quantities above by a constant multiple of $\|f\|_{L^\infty}.$ We will first bound $\mathrm{A}$. Fix $I \in \D$, and write $Tf= T(f \one_{I})+ T(f \one_{I^c}):= T(f_1)+ T(f_2)$. We have, by Jensen's inequality and the $L^2$ boundedness on $T$,

\begin{align*}
 \langle |Tf_1- \langle Tf_1 \rangle_{I}| \rangle_{I} & \leq  \frac{2}{\mu(I)} \int_{I} |T(f_1)| \, d\mu \\
 & \leq 2 \left(\frac{1}{\mu(I)} \int_{I} |T(f_1)|^2 \, d\mu \right)^{1/2} \\
 & \lesssim \left( \frac{1}{\mu(I)} \int_{I} |f|^2 \, d \mu\right)^{1/2} \leq\|f\|_{L^\infty},
\end{align*}
as desired. Next, we control the same piece with $f_1$ replaced by $f_2$. The integrand is
\begin{equation}
\left[ \left(T(f_2)- \langle T(f_2) \rangle_{I}\right) \one_{I}   \right] = \sum_{J \in \D}
\alpha_J \langle f \one_{I^c}, h_{J_m^s} \rangle h_{J_n^t} \one_I - \left( \frac{1}{\mu(I)} \int_I  \left(\sum_{J \in \D}
\alpha_J\langle f \one_{I^c}, h_{J_m^s} \rangle h_{J_n^t}  \right) \, d\mu \right) \one_I \label{IntegrandDisplay}
\end{equation}

In the first sum, by support conditions all the terms vanish except those corresponding to $J$ satisfying $J_n^t \cap I \neq \emptyset$ and $J_s^m \cap I^c \neq \emptyset $. In the second sum, all surviving $J$ satisfy $J_n^t \supsetneq I$ and $J_s^m \cap I^c \neq \emptyset$. Therefore, if we let $x_I$ denote any point belonging to the interval $I$, \eqref{IntegrandDisplay} reduces to

\begin{align*}
& \sum_{\substack{J \in \D: \\ J_n^t \cap I \neq \emptyset,J_s^m \cap I^c \neq \emptyset } }
\alpha_J \langle f \one_{I^c}, h_{J_m^s} \rangle h_{J_n^t} \one_I - \left( \frac{1}{\mu(I)} \int_I  \left(\sum_{\substack{J \in \D: \\ J_t^n \supsetneq I,J_s^m \cap I^c \neq \emptyset } }
\alpha_J  \langle f \one_{I^c}, h_{J_m^s} \rangle h_{J_n^t}  \right) \, d\mu \right) \one_I \\
& =\sum_{\substack{J \in \D: \\ J_t^n \cap I \neq \emptyset,J_s^m \cap I^c \neq \emptyset } } \alpha_J
\langle f \one_{I^c}, h_{J_m^s} \rangle h_{J_n^t} \one_I- \sum_{\substack{J \in \D: \\ J_t^n \supsetneq I,J_s^m \cap I^c \neq \emptyset } } \alpha_J
\langle f \one_{I^c}, h_{J_m^s} \rangle h_{J_n^t}(x_I) \chi(I) \\
& = \sum_{\substack{J \in \D: \\ J_t^n \subseteq I,J_s^m \cap I^c \neq \emptyset } }
\alpha_J \langle f \one_{I^c}, h_{J_m^s} \rangle h_{J_n^t}.
\end{align*}    
If $J$ satisfies both conditions in the final sum, then $J \subseteq I^{(t)}$ (the $t$-th ancestor of $I$), but also $J \cap I \neq \emptyset$ and $J_s^m \not\subseteq I.$ These three conditions force

$$ I \subseteq J \subseteq I^{(t)},$$
and there are at most $t$ intervals intervals $J$ satisfying them. Therefore, 
\begin{align*}
\left \| \sum_{\substack{J \in \D: \\ J_t^n \subseteq I,J_s^m \cap I^c \neq \emptyset } }
\alpha_J \langle f \one_{I^c}, h_{J_m^s} \rangle h_{J_n^t} \right \|_{L^\infty} & \leq t \sup_{J, \, I \subseteq J \subseteq I^{(t)}} |\langle f \one_{I^c}, h_{J_m^s} \rangle| \|h_{J_n^t}\|_{L^\infty}\\
& \lesssim \|f\|_{L^\infty } \sup_{J, \, I \subseteq J \subseteq I^{(t)}} \|h_{J_m^s}\|_{L^1(\mu)} \|h_{J_n^t}\|_{L^\infty}  \lesssim \|f\|_{L^\infty },
\end{align*}
where we used \eqref{eq:balancedDim1} in the final step. This completes the estimate for $\mathrm{A}.$ We now turn to estimating $\mathrm{B}$. By the mean-zero property of the Haar functions, we have 
$$ \langle Tf \rangle_I= \sum_{J: \, J_n^t \supsetneq I} \alpha_J \langle f, h_{J_m^s} \rangle h_{J_n^t}(x_I),$$

$$ \langle Tf \rangle_{\widehat{I}}= \sum_{J: \, J_n^t \supsetneq \widehat{I}} \alpha_J\langle f, h_{J_m^s} \rangle h_{J_n^t}(x_I), $$
for any  $x_I\in I$ (which also belongs to $\widehat{I}$). Thus, the only term that survives is $J$ with $J_n^t=\widehat{I}$ (if there even is such a $J$), and in this case $J_m^s=K$ where $K \in \D$ satisfies $\operatorname{dist}_{\D}(K,\widehat{I}) \leq s+t.$ Therefore, in this case,
$$
|\langle Tf \rangle_I-\langle Tf \rangle_{\widehat{I}}|= |\alpha_J| |\langle f, h_K \rangle h_{\widehat{I}}(x_I)| \leq \|f\|_{L^\infty} \|h_K\|_{L^1(\mu)} \|h_{\widehat{I}}\|_{L^\infty(\mu)} \lesssim \|f\|_{L^\infty},
$$
using again \eqref{eq:balancedDim1} in the last step. 
\end{proof}
\end{thm}

We next obtain the upgraded estimate that will imply that $T$ maps $\BMO$ to $\BMO$. Below, we use the John-Nirenberg inequality (see, for example, \cite{HMW}) to simplify the proof. 

\begin{thm}\label{BMOtoBMO}
If $f\in L^\infty(\mu)$ has compact support, then
\begin{equation*}
\|T f\|_{\BMO} \lesssim \|f\|_{\BMO}.
\end{equation*}
\begin{proof}
By the John-Nirenberg inequality, it is enough to prove
\begin{equation*}
\sup_{I \in \D} \fint_{I} |Tf- \langle Tf \rangle_{\widehat{I}}|^2\, d\mu \leq C  \sup_{I \in \D} \fint_{I} |f- \langle f \rangle_{\widehat{I}}|^2\, d\mu. 
\end{equation*}
To this end, fix $I \in D.$ We expand the integrand as a sum of Haar functions:
\begin{align*}
(Tf- \langle Tf \rangle_{\widehat{I}}) \one_I & = \left(\sum_{J \subseteq \widehat{I}} \langle Tf, h_J \rangle h_J \right) \one_I \\
& =\left(\sum_{J \subseteq \widehat{I}} \sum_{K: K_n^t=J} \alpha_K \langle f, h_{K_m^s} \rangle h_{K_n^t} \right) \one_I\\
& = \sum_{K \subseteq I} \alpha_K \langle f, h_{K_m^s} \rangle h_{K_n^t}+ \sum_{\substack{K \supsetneq I:\\ K_n^t \subseteq \widehat{I}}} \alpha_K \langle f, h_{K_m^s} \rangle h_{K_n^t} \one_I \\
& = \sum_{K \subseteq I} \alpha_K \langle g, h_{K_m^s} \rangle h_{K_n^t}+ \sum_{\substack{K \supsetneq I:\\ K_n^t \subseteq \widehat{I}}} \alpha_K \langle f, h_{K_m^s} \rangle h_{K_n^t} \one_I\\
& := F_1 + F_2,
\end{align*}
where $g= (f- \langle f \rangle_{\widehat{I}}) \one_I.$ By orthogonality of the Haar functions, we have
\begin{align*} 
\fint_{I} |F_1|^2 \, d \mu & \leq  \fint_{I} |g|^2\\
& =\fint_{I} |f- \langle f \rangle_{\widehat{I}}|^2\, d\mu.
\end{align*}
$F_2$ is a summation with at most $t+1$ terms, so it suffices to control the $L^\infty$ norm of a single term by $\|f\|_{\BMO}$. But this can easily be done using \eqref{eq:balancedDim1} as in the proof of Theorem \ref{LInfBMO}.
\end{proof}
\end{thm}

The above results do not imply item (ii) of Theorem \ref{th:theoremA}, because $L^2(\mu)$ is not norm-dense in $\BMO(\mu)$. To extend the definition of $T$ to the whole space, we proceed in the standard way. A Haar shift operator is a kernel operator, meaning that the representation
$$
Tf(x) =\int K(x,y) f(y) d\mu(y)
$$
holds for some kernel $K$ which is locally integrable and $f\in L^2(\mu)$.
In addition, we have $K(x,y)=K(x,z)$ and $K(y,x)=K(z,x)$ whenever $y,z\in I$ and $x \in [I^{(r+s)}]^c$ (say). This means that, for any $g$ supported away from $I^{(r+s)}$, we have
$$
\int_{[I^{(r+s)}]^c} [|K(x,y)-K(x,z)|+|K(y,x)-K(z,x)|] |g(x)|d\mu(x)=0.
$$
Therefore, one can safely define $T$ as follows for any given $f\in \BMO(\mu)$: fix a cube $I_0 \in \D$ that contains the origin and define
$$
Tf(x) = T\left(f \one_{I_0^{(r+s)}}\right)(x),
$$
for $x\in I_0$. One can check that this definition depends only on the choice of $I_0$ (but not on the point $x$), and so it defines $Tf$ as a $\BMO$ function. This definition allows one to upgrade Theorem \ref{BMOtoBMO} to part (ii) of Theorem \ref{th:theoremA}. The standard details are left to the reader.

\subsection{$\Hone$ estimates} To prove our $\Hone$ results, we need to use the atomic block decomposition from \cite{CAP}. Let $1< p< \infty.$ A $\mu$-measurable function $b: \mathbb{R} \rightarrow \R$ is a martingale $p$-atomic block when there exists $k \in \Z$ such that
\begin{itemize}
    \item $\cexp_k b=0$;
    \item $b= \sum_j \lambda_j a_j$, where $\text{supp}(a_j) \subset I_j$, $I_j \in \D_{k_j}$ for $k_j \geq k$ and
    $$ \|a_j\|_{L^p(\mu)} \leq \mu(I_j)^{-1/p'} \left( \frac{1}{k_j-k+1} \right).$$
\end{itemize}
We refer to each $a_j$ as a $p$-subatom. Given a martingale $p$-atomic block, set 
$$|b|^1_{\text{atb},p}= \inf_{\substack{b= \sum_j \lambda_j a_j; \\ a_j \text{ $p$-subatoms}}} \sum_{j \geq 1} |\lambda_j|.  
$$
We define the atomic block Hardy space
$$ \Hone_{\text{atb}}(\mu)= \left\{f \in L^1(\mu): f= \sum_{i} b_i, b_i \text{ martingale 2-atomic block} \right\},$$
equipped with the norm 
$$ \|f\|_{\Hone_{\text{atb}}(\mu)}= \inf_{\substack{f= \sum_i b_i; \\ b_i\text{ $2$-atomic blocks}}} \sum_{i \geq 1} |b_i|^{1}_{\text{atb},2} = \inf_{\substack{f= \sum_i b_i; \\ b_i= \sum_{j} \lambda_{i,j} a_{ij}}} \sum_{i,j \geq 1} |\lambda_{i,j}|. $$
Then, the main result of \cite{CAP} states that 
$$ 
\Hone_{\text{atb}}(\mu) \simeq \Hone(\mu) 
$$
with equivalent Banach space norms. 

\begin{thm}\label{H1L1}
For all $f \in \Hone(\mu)$, there holds 
\begin{equation}
\|T f\|_{L^1(\mu)} \lesssim \|f\|_{\Hone_{\mathrm{atb}}(\mu)}.
  \label{H1L1Eq}  
\end{equation}
\begin{proof}
It is easy to verify that $\|T\|_{\Hone_{\text{atb}}(\mu) \rightarrow L^1(\mu)}< \infty$ is equivalent to proving
\begin{equation} \|T(b)\|_{L^1(\mu)} \leq C |b|^{1}_{\text{atb},2} \label{reduction} \end{equation}
for all $2$-atomic blocks $b$ and $C$ an independent constant. Fix one such $b$. By density, we may assume that it has a finite atomic expansion and thus has a Haar expansion which converges unconditionally in $L^2(\mu).$ Using the fact that there exists $k \in \Z$ with $\cexp_kb=0$, we have, using the definition of expectation and Haar functions,
\begin{align*}
T(b) & = \sum_{\ell} \sum_{J \in \D_{\ell}} \alpha_J \langle b, h_{J_m^s} \rangle h_{J_n^t}\\
& = \sum_{\ell}\sum_{J \in \D_{\ell}} \alpha_J \langle b- \cexp_k b, h_{J_m^s} \rangle h_{J_n^t}\\
& = \sum_{\ell \geq k-s}\sum_{J \in \D_{\ell}} \alpha_J \langle b- \cexp_k b, h_{J_m^s} \rangle h_{J_n^t}\\
& = \sum_{\ell \geq k-s}\sum_{J \in \D_{\ell}} \alpha_J \langle b, h_{J_m^s} \rangle h_{J_n^t}\\ 
& = \sum_{K \in \D_{k-s}} \sum_{J \subseteq K} \alpha_J \langle b, h_{J_m^s} \rangle h_{J_n^t}.
\end{align*}
Since $K \in \D_{k-s}$ are pairwise disjoint, we may estimate the $L^1$-norm of $T(b)$ as follows:
\begin{align*}
\|T(b)\|_{L^1(\mu)} & = \sum_{K \in \D_{k-s}}\left \|  \sum_{J \subseteq K} \alpha_J \langle b, h_{J_m^s} \rangle h_{J_n^t}\right \|_{L^1(\mu)}\\
& = \sum_{K \in \D_{k-s}}\left \| \sum_{j: I_j \subseteq K} \sum_{J \subseteq K} \lambda_j \alpha_J \langle a_j, h_{J_m^s} \rangle h_{J_n^t}\right \|_{L^1(\mu)} \\
& \leq \sum_{K \in \D_{k-s}} \sum_{j: I_j \subseteq K} |\lambda_j| \left \| \sum_{J \subseteq K} \alpha_J  \langle a_j, h_{J_m^s} \rangle h_{J_n^t}\right \|_{L^1(\mu)} 
\end{align*}
If we show 
\begin{equation}
\left \| \sum_{J \subseteq K} \alpha_J  \langle a_j, h_{J_m^s} \rangle h_{J_n^t}\right \|_{L^1(\mu)}  \lesssim 1, \label{Atom Reduction}
\end{equation} with implicit constant independent of $j$, then summing in $j$ and then in $K$ we obtain (using disjointness of $K \in \D_{s-k}$) $$ \|T(b)\|_{L^1(\mu)} \lesssim \sum_{j} |\lambda_j|,$$
after which \eqref{reduction} follows upon taking an infimum over all possible representations $b= \sum_{j} \lambda_j a_j.$ To establish \eqref{Atom Reduction}, we break the sum into two pieces:
\begin{align*}
\left \| \sum_{J \subseteq K} \alpha_J  \langle a_j, h_{J_m^s} \rangle h_{J_n^t}\right \|_{L^1(\mu)}  & \leq \left \| \sum_{\substack{J \subseteq K:\\ J_n^t \subseteq I_j}} \alpha_J \langle a_j, h_{J_m^s} \rangle h_{J_n^t}\right \|_{L^1(\mu)} + \left \| \sum_{\substack{J \subseteq K:\\ J_n^t \supsetneq I_j}} \alpha_J  \langle a_j, h_{J_m^s} \rangle h_{J_n^t}\right \|_{L^1(\mu)} \\
& := \mathrm{A}+ \mathrm{B}.
\end{align*}
For $\mathrm{A}$, note that since all $h_{J_n^t}$ are supported on $I_j$, we can apply H\"{o}lder's inequality to obtain 

\begin{align*}
\mathrm{A} & \lesssim \mu(I_j)^{1/2}\left \| \sum_{\substack{J \subseteq K:\\ J_n^t \subseteq I_j}} \alpha_J  \langle a_j, h_{J_m^s} \rangle h_{J_n^t}\right \|_{L^2(\mu)} \\
& \lesssim \mu(I_j)^{1/2} \|a_j\|_{L^2(\mu)}\\
& \lesssim \mu(I_j)^{1/2} \mu(I_j)^{-1/2} \left(\frac{1}{k_j-k+1} \right)\\
& \leq 1.
\end{align*}
On the other hand for $\mathrm{B}$, for fixed $K \in \D_{k-s}$ and $j$ satisfying $I_j \subseteq K$, if we have $J \subseteq K$ and $J_n^t \supsetneq I_j$, then

$$ I_j^{(t)} \subsetneq J \subseteq K,$$
and a moment's thought shows that there can be at most $k_j-k+s$ intervals $J$ satisfying these constraints. In fact, the number of intervals is less than or equal to $\max\{k_j-k+s-t,0\}$, but we will not need such a precise estimate.  Therefore, we control $\mathrm{B}$ as follows, using the balanced condition on $\mu$:

\begin{align*}
\mathrm{B} & \leq (k_j-k+s) \max_{J:I_j^{(t)} \subsetneq J \subseteq K} \| \alpha_J \langle a_j, h_{J_m^s} \rangle h_{J_n^t}\|_{L^1(\mu)}\\
& \leq (k_j-k+s) \|a_j\|_{L^1(\mu)} \max_{J:I_j^{(t)} \subsetneq J \subseteq K} \|h_{J_m^s}\|_{L^\infty} \|h_{J_n^t}\|_{L^1(\mu)}\\
& \lesssim (k_j-k+s) \mu(I_j)^{1/2} \|a_j\|_{L^2(\mu)}\\
& \lesssim \left( \frac{k_j-k+s}{k_j-k+1} \right) \mu(I_j)^{1/2} \mu(I_j)^{-1/2}\lesssim 1.
\end{align*}
With this estimate the proof is complete. 
\end{proof}
\end{thm}

Next, we upgrade this bound to show that such $T$ actually preserves $\Hone(\mu)$ and complete the proof of Theorem \ref{th:theoremA}. As we already said in the Introduction, by the $\Hone(\mu)-\BMO(\mu)$ duality \cite{Weisz}, this provides another proof of Theorem \ref{BMOtoBMO}.

\begin{thm}\label{H1H1}
For all $f \in \Hone(\mu)$, 
\begin{equation}
\|T f\|_{\Hone} \lesssim \|f\|_{\Hone_{\mathrm{atb}}}.
  \label{H1H1Eq}  
\end{equation}

\begin{proof}
We already know from Theorem \ref{H1L1} that $Tf$ is well-defined as an $L^1(\mu)$ function, so it is enough to establish \eqref{H1H1Eq}. This task reduces, by the sub-linearity of the martingale square function, to showing for any atomic block $b$, $T(b)$ also enjoys the square function estimate
\begin{equation}
\|\mathcal{S}[T(b)]\|_{L^1(\mu)} \leq C \,    |b|^{1}_{\text{atb},2}.
\label{H1H1Reduced}\end{equation}
We may write $\mathcal{S}f$ in terms of the Haar basis:
$$ 
\mathcal{S}f(x)= \left( \sum_{I \in \D} |\langle f, h_I \rangle| ^2 |h_I(x)|^2 \right)^{1/2}.
$$
Expanding $T(b)$ with the above expression, we have
\begin{align*}
\mathcal{S} [T(b)] & = \mathcal{S} \left(\sum_{K \in \D_{k-s}} \sum_{J \subseteq K} \alpha_J \langle b, h_{J_m^s} \rangle h_{J_n^t}\right)\\
& = \left(\sum_{K \in \D_{k-s}} \sum_{J \subseteq K} \alpha_J^2 \langle b, h_{J_m^s} \rangle^2 h_{J_n^t}^2\right)^{1/2}
\end{align*}
Since $K \in \D_{k-s}$ are pairwise disjoint, again we have
\begin{align*}
\|\mathcal{S}[T(b)]\|_{L^1(\mu)} & \leq \sum_{K \in \D_{k-s}}\left \| \left( \sum_{J \subseteq K} \alpha_J^2 \langle b, h_{J_m^s} \rangle^2 h_{J_n^t}^2\right )^{1/2}  \right\|_{L^1(\mu)}\\
& = \sum_{K \in \D_{k-s}}\left \| \left(\sum_{J \subseteq K}  \alpha_J^2 \left\langle \sum_{j: I_j \subseteq K} \lambda_j  a_j, h_{J_m^s} \right\rangle^2 h_{J_n^t}^2 \right)^{1/2}\right \|_{L^1(\mu)} \\
& \leq \sum_{K \in \D_{k-s}} \sum_{j: I_j \subseteq K} |\lambda_j| \left \| \left(\sum_{J \subseteq K} \alpha_J^2  \langle a_j, h_{J_m^s} \rangle^2 h_{J_n^t}^2 \right)^{1/2}\right \|_{L^1(\mu)} 
\end{align*}
and as before, it suffices to show 

$$ \left \| \left(\sum_{J \subseteq K} \alpha_J^2  \langle a_j, h_{J_m^s} \rangle^2 h_{J_n^t}^2 \right)^{1/2}\right \|_{L^1(\mu)} \lesssim 1     $$
We again split, using sub-additivity of the square root function:
\begin{align*}
\left \|\left( \sum_{J \subseteq K} \alpha_J^2  \langle a_j, h_{J_m^s} \rangle^2 h_{J_n^t}^2 \right)^{1/2} \right \|_{L^1(\mu)}  & \leq \left \|\left( \sum_{\substack{J \subseteq K:\\ J_n^t \subseteq I_j}} \alpha_J^2 \langle a_j, h_{J_m^s}^2 \rangle h_{J_n^t}^2\right)^{1/2}\right \|_{L^1(\mu)} \\
& + \left \| \left( \sum_{\substack{J \subseteq K:\\ J_n^t \supsetneq I_j}} \alpha_J^2  \langle a_j, h_{J_m^s} \rangle^2 h_{J_n^t}^2 \right)^{1/2}\right \|_{L^1(\mu)}.
\end{align*}
Each term can be controlled by the same procedure that is applied in Theorem \ref{H1L1}. In particular, the first term on the display above can be estimated using the Cauchy-Schwarz inequality and $L^2$-orthogonality, while the second contains a finite summation and the $L^1$ norm of each individual term (namely, $\|\alpha_J \langle a_j, h_{J_m^s} \rangle h_{J_n^t}\|_{L^1(\mu)}$) can be uniformly controlled using \eqref{eq:balancedDim1}. We omit the precise details as they are so similar to what was done before. 
\end{proof}

\end{thm}

\section{Lipschitz Spaces}\label{sec2}

We need the following simple lemma before proving Theorem \ref{th:theoremB}.

\begin{lemma} \label{SibControl}
Let $I$ a dyadic interval and $\mu$ a locally finite Borel measure satisfying 
$$
\mu(I_{-})>0, \quad \mu(I_{+}) \geq \frac{1}{2} \mu(I).
$$
Then for any locally integrable function $f$, we have 
\begin{equation}\label{siblingcontrol} \left| \langle f \rangle_{I_{-}}- \langle f \rangle_{I_{+}}\right| \leq 2 \left| \langle f \rangle_{I_{-}}-\langle f \rangle_{I} \right|.\end{equation}
The analogous result also holds true with the roles of $I_{-}, I_{+}$ reversed.
\begin{proof}
Set $$ \alpha_1= \int_{I_{-}} f \, du, \quad \alpha_2= \int_{I_{+}} f \, du, \quad \alpha=\alpha_1 + \alpha_2= \int_{I} f \, d\mu; $$
$$ \beta_1= \mu(I_{-}), \quad \beta_2= \mu(I_{+}), \quad \beta=\beta_1 + \beta_2= \mu(I). $$
Then \eqref{siblingcontrol} reads

$$ \left| \frac{\alpha_1}{\beta_1}- \frac{\alpha_2}{\beta_2} \right| \leq 2 \left|\frac{\alpha_1}{\beta_1}- \frac{\alpha}{\beta} \right|.   $$
Simple algebra yields the equivalent inequality

$$\left|\frac{\beta}{\beta_2} \right|  \left| \alpha_1 \beta_2-\alpha_2 \beta_1 \right|     \leq 2 \left| \alpha_1 \beta- \alpha \beta_1 \right|  ,$$
and expanding $\alpha,\beta$ this becomes

$$\left|\frac{\beta}{\beta_2} \right|  \left| \alpha_1 \beta_2-\alpha_2 \beta_1 \right|     \leq 2 \left| \alpha_1 \beta_2- \alpha_2 \beta_1 \right|.$$
This inequality clearly holds true if $\left| \alpha_1 \beta_2- \alpha_2 \beta_1 \right|=0$, and otherwise it holds because $|\beta| \leq 2 |\beta_2|.$
\end{proof}

\end{lemma}

\begin{proof}[Proof of Theorem \ref{th:theoremB}]
Since $\mu$ is balanced, $T$ is bounded on $L^q(\mu)$. Let $f \in  \Lip$ and fix an interval $I \in \D.$ We need to control the expression
 
 $$
 \frac{1}{\mu(I)^{1/q+\alpha}} \left( \int_I |Tf- \langle Tf \rangle_{\widehat{I}}|^q \, d\mu \right)^{1/q}.
$$ 
We follow an argument very similar to Theorem \ref{BMOtoBMO}. We have the exact same expansion of the integrand:
\begin{align*}
(Tf- \langle Tf \rangle_{\widehat{I}}) \one_I & = \left(\sum_{J \subseteq \widehat{I}} \langle Tf, h_J \rangle h_J \right) \one_I \\
& = \sum_{K \subseteq I} \alpha_K \langle g, h_{K_m^s} \rangle h_{K_n^t}+ \sum_{\substack{K \supsetneq I:\\ K_n^t \subseteq \widehat{I}}} \alpha_K \langle f, h_{K_m^s} \rangle h_{K_n^t} \one_I\\
& := F_1 + F_2,
\end{align*}
where $g= (f- \langle f \rangle_{\widehat{I}}) \one_I.$ Notice that $F_1= T_I g$, where the Haar shift
 $$ T_I f:= \sum_{K \subseteq I} \alpha_K \langle f, h_{K_m^s} \rangle h_{K_n^t}.$$
There exists a universal constant $C_q$, independent of the dyadic interval $I$, so that for any such Haar shift $T$ with uniformly bounded coefficients, $$\|T_I\|_{L^q(\mu) \rightarrow L^q(\mu)} \leq C_q.$$
The balanced condition on $\mu$ is needed in this step. Therefore, using the $L^q$ bound on $T_I,$ we get 

 \begin{align*}
 \frac{1}{\mu(I)^{1/q+\alpha}} \left( \int_I |F_1|^q \, d\mu \right)^{1/q} & = \frac{1}{\mu(I)^{1/q+\alpha}} \left( \int_I |T_I g|^q \, d\mu \right)^{1/q} \\
 & \lesssim \frac{1}{\mu(I)^{1/q+\alpha}} \left( \int_I | g|^q \, d\mu \right)^{1/q}\\
 & =  \frac{1}{\mu(I)^{1/q+\alpha}} \left( \int_I |f- \langle f \rangle_{\widehat{I}}|^q \, d\mu \right)^{1/q}\\
 & \leq \|f\|_{\Lip}.
 \end{align*}
On the other hand, the sum defining $F_2$ is finite as before, and it suffices to show 

$$\frac{1}{\mu(I)^\alpha} \sup_{{\substack{K \supsetneq I:\\ K_n^t \subseteq \widehat{I}}}} \|\alpha_K \langle f, h_{K_m^s} \rangle h_{K_n^t}\|_{L^\infty} \lesssim \|f\|_{\Lip}.$$
Fix such a $K$. Then by \eqref{eq:balancedDim1} and the relationship between $K$ and $I$, $m(I) \sim m(K_m^s)$ and we may assume without loss of generality $m(K_m^s) \sim \mu({K_m^s}_{-})$. Then, we estimate using Lemma \ref{SibControl}:

\begin{align*}
\frac{1}{\mu(I)^\alpha} \|\alpha_K \langle f, h_{K_m^s} \rangle h_{K_n^t}\|_{L^\infty} & \lesssim  \frac{1}{\mu(I)^\alpha} \left| \langle f \rangle_{{K_m^s}_{-}} -\langle f \rangle_{{K_m^s}_{+}}  \right|\\
& \leq\frac{1}{m(I)^\alpha} \left| \langle f \rangle_{{K_m^s}_{-}} -\langle f \rangle_{{K_m^s}_{+}}  \right|\\
&  \sim \frac{1}{m(K_m^s)^\alpha} \left| \langle f \rangle_{{K_m^s}_{-}} -\langle f \rangle_{{K_m^s}_{+}}  \right| \\
& \sim \frac{1}{\mu({K_m^s}_{-})^\alpha} \left| \langle f \rangle_{{K_m^s}_{-}} -\langle f \rangle_{{K_m^s}_{+}}  \right|\\
& \lesssim \frac{1}{\mu({K_m^s}_{-})^\alpha} \left| \langle f \rangle_{{K_m^s}_{-}} -\langle f \rangle_{{K_m^s}}  \right|\\
& \lesssim \|f\|_{\Lip}.
\end{align*}
\end{proof}

As we mentioned in the Introduction, one might wonder if the balanced condition is really necessary and if one can actually get a result for a generic Borel measure which parallels the continuous case in \cite{GG}. We prove below that the hypothesis on the measure is sharp and that the Lipschitz norm actually characterizes \eqref{eq:balancedDim1}. We take $q=2$ below for convenience. 

\begin{prop} \label{LipshitzCounter}
Let $\mu$ a locally finite positive Borel measure satisfying $\mu(I)>0$ for all $I \in \D.$ Then if all dyadic shifts $T$ are uniformly bounded on $\Lambda_2(\alpha)$, $\mu$ must be balanced. The same is true if we replace $\Lambda_2(\alpha)$ by $\BMO(\mu)$.

\begin{proof} 
Let $T$ be the dyadic shift $Tf:= \sum_{I \in \D} \langle f, h_I \rangle h_{I_{-}}.$ Then for any fixed dyadic $I$, $T h_I= h_{I_{-}}$, and we need to compute $\|h_I\|_{\Lambda_{2}(\alpha)}$ (and at the same time $\|h_{I_{-}}\|_{\Lambda_{2}(\alpha)}$). Note if $J \subsetneq I_{-}$ or $J \subsetneq I_{+}$, the Lipschitz oscillation ($L^2$ average of $|f-\langle f \rangle_{\widehat{J}}|$) over $J$ is equal to 0. If $J \supseteq I$, it is easy to see that by the $L^2$ normalization and mean zero property of $h_I$ that 
$$\frac{1}{\mu(J)^{1/2+\alpha}} \left( \int_J |h_I- \langle h_I \rangle_{\widehat{J}}|^2 \, d\mu \right)^{1/2}= \frac{1}{\mu(J)^{1/2+\alpha}}. $$
The only remaining intervals to test on are $I_{-}$ and $I_{+}$. In either case, it is easy to check that we get, respectively 

$$\frac{1}{\mu(I_{\pm})^{1/2+\alpha}} \left( \int_{I_{\pm}} |h_I- \langle h_I \rangle_{I}|^2 \, d\mu \right)^{1/2}= \frac{m(I)^{1/2}}{\mu(I_{\pm})^{1+\alpha}}$$
Collectively, we have shown $\|h_I\|_{\Lambda_{2}(\alpha)} \sim m(I)^{-\alpha-1/2}.$ Then, the uniform boundedness condition imposes $\|h_{I_{-}}\|_{\Lambda_{2}(\alpha)} \lesssim \|h_I\|_{\Lambda_{2}(\alpha)}$, which further implies

$$m(I_{-})^{-\alpha-1/2} \lesssim m(I)^{-\alpha-1/2}, \text{ or } \quad m(I) \lesssim m(I_{-}).     $$
It is then obvious that we can also obtain the inequalities $m(I) \lesssim m(I_{+})$ and $m(I) \lesssim m(\widehat{I})$ by testing on analogous shifts. This produces the balanced condition. To prove the result for $\BMO(\mu)$, one can just repeat the argument above with $\alpha=0.$
\end{proof}

\begin{rem}
Actually, the proof shows that we could weaken the condition that \emph{all} Haar shifts are uniformly bounded to merely assuming that the classical dyadic Hilbert transform
$$ Hf:=  \sum_{I \in \D} \langle f, h_I \rangle (h_{I_{-}}-h_{I_{+}})$$ and its adjoint are both bounded. 

\end{rem}

\end{prop}

We close this section with a couple of open questions which arise naturally from our investigation of Lipschitz spaces. 

\begin{open}
Do the spaces $\Lip$ depend on the exponent $q$? More generally, is there an analog of the John-Nirenberg inequality for martingale BMO for these Lipschitz spaces?
    
\end{open}

\begin{open}
Does Theorem \ref{th:theoremB} hold with exponent $q=1$? Obviously, if the Lipschitz spaces $\Lip$ are independent of the exponent $q$, then this will be the case. 
 \end{open}

\begin{open}
Identify the pre-dual of Lipschitz spaces $\Lip$, in analogy with the isometric isomorphism $\BMO \simeq \Hone_{atb}(\mu).$
\end{open}

\bibliographystyle{alpha}
\bibliography{references}

\end{document}